\documentclass{article}
\usepackage{amsmath,amsxtra,amssymb,amsthm,amsfonts}
\usepackage{bbm}
\usepackage{mathtools}
\usepackage{color,graphicx}
\usepackage[margin=1.05in]{geometry}
\usepackage[dvipsnames]{xcolor}
\usepackage{etoolbox}
\usepackage{url}

\usepackage[linktocpage=true,colorlinks,citecolor=blue,linkcolor=blue,pagebackref]{hyperref}
\usepackage{cleveref}
\usepackage{caption}
\usepackage{subcaption}

\usepackage{pgfplots}
\usepackage{pgf}
\usepackage{tikz}
\usetikzlibrary{patterns}
\usetikzlibrary{arrows.meta}
\usepgfplotslibrary{patchplots} % LATEX and plain TEX
\usetikzlibrary{pgfplots.patchplots} % LATEX and plain TEX
\pgfplotsset{width=9cm,compat=1.5.1}

\usepackage{enumitem}
\usepackage{tikz-cd}
\usetikzlibrary{positioning}
\usetikzlibrary{calc}

\newcommand\tr{\mathop{\rm tr}\nolimits}

\makeatletter
\newcommand{\xbij}[2][]{%
  \lhook\joinrel
  \ext@arrow 0359\rightarrowfill@ {#1}{#2}%
  \mathrel{\mspace{-15mu}}\rightarrow
}
\makeatother

\usepackage{algorithm,algpseudocode}
\algnewcommand\algorithmicforeach{\textbf{for each}}
\algdef{S}[FOR]{ForEach}[1]{\algorithmicforeach\ #1\ \algorithmicdo}

\algnewcommand{\algorithmicand}{\textbf{ and }}
\algnewcommand{\algorithmicor}{\textbf{ or }}
\algnewcommand{\algorithmicnot}{\textbf{ not }}
\algnewcommand{\algorithmicbreak}{\textbf{break}}
\algnewcommand{\algorithmiccontinue}{\textbf{continue}}
\algnewcommand{\AND}{\algorithmicand}
\algnewcommand{\OR}{\algorithmicor}
\algnewcommand{\NOT}{\algorithmicnot}
\algnewcommand{\BREAK}{\algorithmicbreak}
\algnewcommand{\CONTINUE}{\algorithmiccontinue}

\algrenewcommand\algorithmiccomment[1]{\hfill{\color{gray}\(\triangleright\)
#1}}

\makeatletter
\newcommand{\LeftComment}[1]{%
  \Statex
  {%
    \setlength\leftskip{\ALG@thistlm}%
    \noindent
    \color{gray}\(\triangleright\)\enspace#1\par%
  }%
}

\newcommand{\LeftCommentIndent}[1]{%
  \Statex
  {%
    \setlength\leftskip{\dimexpr\ALG@thistlm+\algorithmicindent}%
    \noindent
    \color{gray}\(\triangleright\)\enspace#1\par%
  }%
}
\makeatother

\def\C{\mathbb{C}}

\def\Z{\mathbb{Z}}

\def\R{\mathbb{R}}

\def\F{\mathbb{F}}

\def\Ein{\mathcal{E}}
\def\K{\mathcal{K}}

\def\e{\mathbf{e}}

\def\u{\mathbf{u}}
\def\vv{\mathbf{v}} % To avoid overriding `\v` (caron diacritic)

\def\0{\mathbf{0}}
\def\1{\mathbf{1}}

\numberwithin{equation}{section}
\theoremstyle{plain}

\newtheorem{theorem}{Theorem}[section]
\newtheorem{proposition}[theorem]{Proposition}
\newtheorem{lemma}[theorem]{Lemma}

\newtheorem{remark}[theorem]{Remark}

\newtheorem{algo}[theorem]{Algorithm}

\newtheorem{definition}{Definition}

\begin{document}

\title{
  Efficient spectral bounds on the chromatic number of Hamming,
  Johnson, and Kneser graph powers
}

\author{
  Finn A. Steinke\textsuperscript{1} and Luis M. B.
  Varona\textsuperscript{2,3,4}
}

\maketitle

\begin{abstract}
  We investigate spectral lower bounds on the chromatic number $\chi$
  of Hamming graph powers $H(n, q)^p$, Johnson graph powers $J(n,
  k)^p$, and Kneser graph powers $K(n, k)^p$ providing the first
  computationally feasible nontrivial results. While the classical
  Hoffman bound on $\chi$ can, in principle, be applied to any graph,
  na\"ive computation requires $O(q^{3n})$ time for $H(n, q)^p$ and
  $O(({}_nC_k)^3)$ time for both $J(n, k)^p$ and $K(n, k)^p$. We thus
  express the adjacency eigenvalues of these graphs in terms of
  hypergeometric orthogonal polynomials, exploiting recurrence
  relations that arise to efficiently compute the entire spectra. We
  then apply dynamic programming to compute the Hoffman bounds for
  $H(n, q)^p$, $J(n, k)^p$, and $K(n, k)^p$ in $O(np)$, $O(kp)$, and
  $O(k^2)$ time, respectively.\\\medskip


\end{abstract}

\addtocounter{footnote}{1}
\footnotetext{Department of Mathematics \& Statistics, University of
Guelph, Guelph, ON, Canada N1G 2W1}
\addtocounter{footnote}{1}
\footnotetext{Department of Politics \& International Relations,
Mount Allison University, Sackville, NB, Canada E4L 1E4}
\addtocounter{footnote}{1}
\footnotetext{Department of Economics, Mount Allison University,
Sackville, NB, Canada E4L 1E4}
\addtocounter{footnote}{1}
\footnotetext{Department of Mathematics \& Computer Science, Mount
Allison University, Sackville, NB, Canada E4L 1E4}

\section{Introduction}\label{section:1}

The \emph{chromatic number} $\chi(G)$ of a graph $G$ is the minimum number of colours needed to colour the vertices of $G$ such that no two adjacent vertices share the same colour. Determining chromatic number has widespread applications ranging from scheduling to resource allocation; although exact computation is NP-hard in the number of vertices, spectral methods can provide valuable lower bounds. The well-known Hoffman bound \cite{Hof03} states that $\chi(G) \ge 1 + \frac{\lambda_{\text{max}}}{\lvert \lambda_{\text{min}} \rvert}$, where $\lambda_{\text{max}}$ and $\lambda_{\text{min}}$ are the maximum and minimum eigenvalues, respectively, of $G$'s adjacency matrix. However, finding the adjacency eigenvalues of a graph on $n$ vertices is typically $O(n^3)$, unless the graph possesses additional algebraic structure that can be exploited.

Hamming, Johnson, and Kneser graphs---denoted by $H(n, q)$, $J(n, k)$, and $K(n, k)$, respectively---are significant examples of highly structured, vertex-transitive graphs. Many of these structural properties extend to graph powers of the form $H(n, q)^p$, $J(n, k)^p$, and $K(n, k)^p$, making it easier to compute their respective Hoffman bounds. Studying colourings of these graphs has practical applications in coding theory, error correction, distributed computing, and more---for instance, in models that parametrize processes by $n$ attributes with $q$ values each, $\chi(H(n, q)^p)$ is the minimum number of resources required for conflict-free allocation when processes whose configurations differ in at most $p$ parameters conflict.

To our knowledge, no results on chromatic lower bounds for Hamming, Johnson, or Kneser graph powers exist, with the exception of hypercube powers $Q_n^p = H(n, 2)^p$. \cite{FFR17} conducted the most prominent research in this field to date, establishing an exact expression for the clique size $\omega$ of $Q_n^p$ that is computable in $O(p)$ time; since $\omega \le \chi$ for all graphs, this extends to a lower bound on $\chi(Q_n^p)$. Nevertheless, Johnson and Kneser graph powers remain unexplored, as do Hamming graph powers in the general case. We therefore aim to establish the first computationally feasible chromatic lower bounds for these graphs.

While the aforementioned Hoffman bound can be applied to any graph, general-purpose methods for finding eigenvalues---most prominently singular value decomposition (SVD) and QR decomposition (QRD)---are cubic in complexity. Consequently, na\"ive computation requires $O(q^{3n})$ time for $H(n, q)^p$ (which has $q^n$ vertices) and $O\bigl(\binom{n}{k}^3\bigr)$ time for both $J(n, k)^p$ and $K(n, k)^p$ (which have $\binom{n}{k}$ vertices). To overcome this bottleneck, we seek to exploit properties of corresponding algebraic structures known as association schemes, which we define later in Subsection \ref{subsection:2.3}.

We are thus able to utilize dynamic programming to compute the Hoffman bounds for $H(n, q)^p$, $J(n, k)^p$, and $K(n, k)^p$ in $O(np)$, $O(kp)$, and $O(k^2)$ time, respectively. This represents a dramatic reduction in complexity: for instance, na\"ively bounding $\chi(H(n, q)^p)$ from below via the Hoffman bound would require computing at least the minimum and maximum eigenvalues of a $q^n \times q^n$ matrix using methods like SVD or QRD, whereas Algorithm \ref{algo:3.5} achieves the same result merely by performing $p(2n + 2)$ constant-time array updates. Even for modest parameter values such as $(n, q, p) = (10, 3, 5)$, this amounts to just $110$ $O(1)$ operations rather than computing the spectrum of a $59,\!049 \times 59,\!049$ matrix. Algorithms \ref{algo:4.7} and \ref{algo:4.8} offer similarly appreciable gains over the na\"ive methods for computing the Hoffman bounds for Johnson and Kneser graph powers, respectively.

\subsection{Organization of the paper}\label{subsection:1.1}

We begin in Section \ref{section:2} by introducing the notation used throughout the paper and covering preliminary mathematical definitions. Furthermore, we state established results giving the exact values of $\chi$ for $H(n, q)$ and $K(n, k)$ (but not their powers), then review some known chromatic lower bounds for Johnson graphs $J(n, k)$ and hypercube powers $Q_n^p$. Additionally, we briefly discuss some applications of colourings of Hamming graphs, Johnson graphs, Kneser graphs, and their powers to coding theory, error correction, and distributed computing, lending further credence to the relevance of our results.

In Section \ref{section:3}, we find a closed form for the adjacency eigenvalues of Hamming graph powers using Kravchuk polynomials, then leverage recurrence relations between these polynomials to establish an $O(np)$ dynamic programming algorithm to compute the Hoffman bound for $H(n, q)^p$. In Section \ref{section:4}, we use Eberlein polynomials to develop analogous $O(kp)$ and $O(k^2)$ algorithms to compute the Hoffman bounds for $J(n, k)^p$ and $K(n, k)^p$, respectively.

Finally, we close in Section \ref{section:5} with some open questions and suggestions for future work. In particular, we consider applications of our eigenvalue computation techniques to other spectral bounds and graph invariants, as well as extensions to other highly symmetric, vertex-transitive graph families.

\section{Preliminaries and definitions}\label{section:2}

We use bold lowercase letters like $\u, \vv, \ldots$ to denote vectors in $\F^n$ for any field $\F$ and, more generally, to denote elements of free modules $R^n$ for any ring $R$ (e.g., $\Z_q^n$). We use $\e_i$ to denote the $i$\textsuperscript{th} standard basis vector (i.e., the vector with $1$ in its $i$\textsuperscript{th} entry and $0$ everywhere else) and $\1$ to denote the all-ones vector $(1, 1, \ldots, 1)$. We often write $\mathrm{Spec}(A)$ for the set of eigenvalues of a matrix $A \in \R^{n \times n}$, $[n]$ for the set of integers $\{1, 2, \ldots, n\}$, and $\binom{[n]}{k}$ for the set of all $k$-subsets of $[n]$. We adopt the standard combinatorial convention that $\binom{n}{k} = 0$ whenever $n < k$.

We typically index matrices and vectors from $1$, using $[A]_{i,j}$ to denote the $(i, j)$\textsuperscript{th} entry of the matrix $A$ and $v_i$ to denote the $i$\textsuperscript{th} entry of the vector $\vv$. However, we follow the standard convention of $0$-based indexing when describing data structures and algorithms---for instance, $\texttt{arr}[i]$ denotes the $(i + 1)$\textsuperscript{th} entry of the array $\texttt{arr}$.

Throughout this paper, we exclusively consider unweighted, undirected, finite graphs with no self-loops. Given a graph $G$, we frequently write $V(G)$ for the vertex set of $G$, $E(G)$ for the edge set of $G$, $A(G)$ for the adjacency matrix of $G$, $\mathrm{diam}(G)$ for the diameter of $G$, $\chi(G)$ for the chromatic number of $G$, and $\omega(G)$ for the clique number of $G$. We are also concerned with powers of graphs, which we define below.

\begin{definition}[Graph powers]\label{defn:1}
    The $p$\textsuperscript{th} \emph{power} of a graph $G$, denoted by $G^p$, is a graph with the same vertex set in which two vertices are adjacent if and only if their distance in $G$ is at most $p$; that is,
    \[V(G^p) \coloneqq V(G), \quad E(G^p) \coloneqq \bigl\{\{u, v\} \subseteq V(G) \bigm| 1 \le d_G(u, v) \le p\bigr\},\]
    where $d_G : V(G) \times V(G) \to \Z_{\ge 0}$ denotes the vertex distance function for $G$.
\end{definition}

\begin{remark}\label{remark:2.1}
    It is helpful to observe that given any graph $G$ on $n$ vertices, the power $G^p$ is simply the complete graph $K_n$ on $n$ vertices for all $p \ge \mathrm{diam}(G)$. As such, we only concern ourselves with graph powers $G^p$ for $1 \le p \le \mathrm{diam}(G)$ throughout this paper.
\end{remark}

\subsection{Hamming graphs}\label{subsection:2.1}

We now turn to understanding Hamming graphs, as well as the specific case of the hypercube. (The interested reader may wish to refer to \cite[pp. 261--267]{BCN89} and \cite[pp. 167--168]{HIK11} for further discussion on Hamming graphs. A nice overview of hypercubes specifically is given in \cite[pp. 17--18]{HIK11} as well.) Before defining these graphs, we first introduce their vertex distance function for convenience---the Hamming metric.

\begin{definition}[Hamming metric and distance]\label{defn:2}
    Let $n \ge 1$ and $q \ge 2$, and consider the set $\Z_q^n$ of length-$n$ strings in a $q$-ary alphabet. Equipped with the \emph{Hamming metric} $d_H : \Z_q^n \times \Z_q^n \to \{0, 1, \ldots, n\}$, this forms the metric space $(\Z_q^n, d_H)$, where we define the \emph{Hamming distance} $d_H(\u, \vv)$ between two strings $\u, \vv \in \Z_q^n$ as the number of positions at which they differ:
    \[d_H(\u, \vv) \coloneqq \lvert \{i \in [n] \mid u_i \ne v_i\} \rvert.\]
\end{definition}

We are now in full possession of the terminology needed to give full presentations of Hamming graphs and hypercubes in terms of their vertex and edge sets.

\begin{definition}[Hamming graph]\label{defn:3}
    Let $n \ge 1$ and $q \ge 2$. The \emph{Hamming graph} $H(n, q)$ is the graph whose vertices correspond to all possible $q$-ary strings of length $n$ and in which two vertices are connected if they differ in precisely one position. More formally, the vertex and edge sets of $H(n, q)$ are given by
    \[V(H(n, q)) \coloneqq \Z_q^n, \quad E(H(n, q)) \coloneqq \bigl\{\{\u, \vv\} \subseteq \Z_q^n \bigm| d_H(\u, \vv) = 1\bigr\}.\]
\end{definition}

\begin{definition}[Hypercube]\label{defn:4}
    The \emph{hypercube} is a special case of a Hamming graph, defined as $Q_n \coloneqq H(n, 2)$. An alternative presentation of $Q_n$, independent of the terminology of Hamming graphs, is
    \[V(Q_n) \coloneqq \Z_2^n, \quad E(Q_n) \coloneqq \bigl\{\{\u, \vv\} \subseteq \Z_2^n \bigm| d_H(\u, \vv) = 1\bigr\}.\]
\end{definition}

Note that the diameter of the Hamming graph is $\mathrm{diam}(H(n, q)) = n$ regardless of the value of $q$, since the maximum Hamming distance between any two vertices is $n$ (when their associated $q$-ary strings differ in all positions). As per Remark \ref{remark:2.1}, we therefore only consider Hamming graph powers $H(n, q)^p$ for $p \le n$.

\subsection{Johnson and Kneser graphs}\label{subsection:2.2}

We are also interested in establishing chromatic lower bounds for Johnson graphs, Kneser graphs, and their powers. (Once again, helpful complementary discussion on Johnson graphs may be found in \cite[pp. 255--261]{BCN89}, which also briefly mentions the closely related Kneser graphs. For a more detailed review of Kneser graphs, consult \cite[p. 22]{HIK11}.) We define these graphs below.

\begin{definition}[Johnson graph]\label{defn:5}
    Let $k \ge 1$ and $n \ge k$. The \emph{Johnson graph} $J(n, k)$ is the graph whose vertices correspond to all possible $k$-subsets of $[n]$ and in which two vertices are adjacent if their intersection has cardinality $k - 1$. More formally, the vertex and edge sets of $J(n, k)$ are given by
    \[V(J(n, k)) \coloneqq \binom{[n]}{k}, \quad E(J(n, k)) \coloneqq \left\{\{A, B\} \subseteq \binom{[n]}{k}\, \middle|\, \lvert A \cap B \rvert = k - 1\right\}.\]
\end{definition}

\begin{definition}[Kneser graph]\label{defn:6}
    Let $k \ge 1$ and $n \ge 2k$. The \emph{Kneser graph} $K(n, k)$ is the graph whose vertices correspond to all possible $k$-subsets of $[n]$ and in which two vertices are adjacent if they are disjoint sets. More formally, the vertex and edge sets of $K(n, k)$ are given by
    \[V(K(n, k)) \coloneqq \binom{[n]}{k}, \quad E(K(n, k)) \coloneqq \left\{\{A, B\} \subseteq \binom{[n]}{k}\, \middle|\, A \cap B = \emptyset\right\}.\]
\end{definition}

\begin{remark}\label{remark:2.2}
    The standard definition of the Kneser graph imposes the constraint $n \ge 2k$ because if $n < 2k$, then no two $k$-subsets of $[n]$ are disjoint, and $K(n, k)$ is simply the empty graph with no edges. Moreover, in the $n = 2k$ case, $K(n, k)^p$ is a perfect matching for every $p \ge 1$. Perfect matchings are well known to have a chromatic number of $2$, so we are only interested in investigating the Hoffman bound on $\chi(K(n, k)^p)$ for $n > 2k$, even though $n = 2k$ is a valid case according to the definition.
\end{remark}

It is worth noting that Johnson and Kneser graphs are closely related: $K(n, k)$ is precisely the distance-$k$ graph of $J(n, k)$. That is, $K(n, k)$ and $J(n, k)$ have the same vertex sets, and two vertices in $K(n, k)$ are adjacent if and only if their shortest-path distance in $J(n, k)$ is precisely $k$ (which occurs when the corresponding $k$-subsets of $[n]$ are disjoint). This relationship shall allow us to utilize the same family of HOPs---namely, the Eberlein polynomials---to efficiently compute the Hoffman bounds on both $J(n, k)^p$ and $K(n, k)^p$ in Section \ref{section:4}.

As with the Hamming metric on Hamming graphs, it is straightforward to obtain a closed-form expression for the vertex distance function on $J(n, k)$.

\begin{proposition}\label{prop:2.3}
    Let $k \ge 1$ and $n \ge k$, and suppose that $A, B \in \binom{[n]}{k}$ with $\lvert A \cap B \rvert = x$. The shortest-path distance between $A$ and $B$ in the Johnson graph $J(n, k)$ is then given by
    \[d_{J(n, k)}(A, B) = k - x.\]
\end{proposition}

\begin{proof}
    Since two vertices in $J(n, k)$ are adjacent if and only if they differ by exactly one element, a shortest path from $A$ to $B$ corresponds to successively swapping elements in $A \setminus B$ for elements in $B \setminus A$. Therefore, we require exactly $k - \lvert A \cap B \rvert$ swaps to reach $B$, as desired.
\end{proof}

Hence, the diameter of the Johnson graph is $\mathrm{diam}(J(n, k)) = \min\{k, n - k\}$, since the maximum distance between any two vertices $A$ and $B$ occurs when $\lvert A \cap B \rvert$ is minimized---either $0$ when $n \ge 2k$ (giving distance $k$), or $2k - n$ when $n < 2k$ (giving distance $n - k$). We thus only consider Johnson graph powers $J(n, k)^p$ for $p \le \min\{k, n - k\}$, adhering to Remark \ref{remark:2.1}.

Characterizing adjacency in Kneser graphs, on the other hand, is not as simple. Although there still exists an explicit $O(1)$ formula for the vertex distance function on $K(n, k)$, it is far more complicated and does not follow as directly from the fundamental definition of the Kneser graph. We state this formula---taken from existing literature---below.

\begin{proposition}\label{prop:2.4}
    Let $k \ge 1$ and $n > 2k$, and suppose that $A, B \in \binom{[n]}{k}$ with $\lvert A \cap B \rvert = x$. The shortest-path distance between $A$ and $B$ in the Kneser graph $K(n, k)$ is then given by
    \[d_{K(n, k)}(A, B) = \min\left\{2{\left\lceil\frac{k - x}{n - 2k}\right\rceil},\, 2{\left\lceil\frac{x}{n - 2k}\right\rceil} + 1\right\}.\]
\end{proposition}

\begin{proof}
    Given in \cite{VPV05}.
\end{proof}

Another pertinent result from \cite{VPV05} (indeed, their main contribution) is that the diameter of the Kneser graph is precisely $\mathrm{diam}(K(n, k)) = \left\lceil\frac{k - 1}{n - 2k}\right\rceil + 1$. As such, Remark \ref{remark:2.1} again informs us that we need only consider Kneser graph powers $K(n, k)^p$ for $p \le \left\lceil\frac{k - 1}{n - 2k}\right\rceil + 1$.

\subsection{Association schemes}\label{subsection:2.3}

A closely related concept to distance-regular graphs like Hamming and Johnson graphs is that of association schemes, whose properties form the foundation of our methods to efficiently compute Hoffman bounds. (Further discussion of association schemes can be found in \cite{BI84}, which we in fact reference several times throughout the text.) We define these objects below.

\begin{definition}[Association scheme]\label{defn:7}
    An $n$-class association scheme is an ordered pair $\bigl(X, \{R_i\}_{i = 0}^n\bigr)$, where $X$ is a set and $\{R_i\}_{i = 0}^n$ is a partition of $X \times X$ that satisfies the following properties:
    \begin{enumerate}[label=(\roman*)]
        \item $R_0 = \bigl\{(x, x) \bigm| x \in X\bigr\}$
        \item For all $1 \le i \le n$, the inverse relation $\bigl\{(x, y) \bigm| (y, x) \in R_i\bigr\}$ is also in the collection $\{R_i\}_{i = 0}^n$
        \item For all $0 \le i, j, k \le n$, there exists a constant $p_{i,j}^k$ such that for any $(x, y) \in R_k$, the number of elements $z \in X$ satisfying $(x, z) \in R_i$ and $(z, y) \in R_j$ is precisely $p_{i,j}^k$
    \end{enumerate}
\end{definition}

Given an association scheme $\bigl(X, \{R_i\}_{i = 0}^n\bigr)$, each $(X, R_i)$ can be viewed as a graph whose vertex set is $X$ and whose edge set is $R_i$. (If the $R_i$'s are not symmetric, this is a directed graph; we therefore restrict our attention in this paper to association schemes whose relations are all symmetric.) As such, every distance-regular graph $G$ has a corresponding association scheme $\bigl(V(G), \{D_i\}_{i = 0}^{\mathrm{diam}(G)}\bigr)$, where the adjacency matrix $A_i$ of each $(V(G), D_i)$ is precisely the distance-$i$ matrix of $G$. It follows from this that the eigenvalues of the graph power $G^p$ are given by summing the eigenvalues of the $A_i$'s over $i = 1, 2, \ldots, p$.

These adjacency matrices are called the \emph{basis matrices} of the association scheme, and their span is called the \emph{Bose--Mesner algebra} of the association scheme. It is an important property that the basis matrices $\{A_i\}_{i = 0}^n$ of any association scheme are simultaneously diagonalizable with a common eigenspace decomposition $V = V_0 \oplus V_1 \oplus \cdots \oplus V_n$, where $V_0$ is the subspace spanned by the all-ones vector \cite[pp. 58--60]{BI84}. Moreover, basis matrices of association schemes that represent distance-regular graphs always have eigenvalues given by a class of polynomials known as \emph{hypergeometric orthogonal polynomials}, or HOPs. These are orthogonal sequences of polynomials that satisfy recurrence relations, making it efficient to compute successive terms---further details can be found in \cite{KS98}. We can thus use the HOPs of the Hamming and Johnson schemes to compute the eigenvalues of Hamming and Johnson graph powers (and Kneser graph powers as well, by their connection to Johnson graphs), which we shall do in Sections \ref{section:3} and \ref{section:4}, respectively.

\subsection{Known bounds and applications}\label{subsection:2.4}

The exact chromatic numbers of both Hamming graphs and Kneser graphs are well known. Since $H(n, q)$ is the Cartesian product of $n$ complete graphs on $q$ vertices, and the chromatic number of a Cartesian product equals the maximum chromatic number among its factors, it is straightforward to see that $\chi(H(n, q)) = q$ for all $n \ge 1$ and $q \ge 2$. Additionally, the foundational work of \cite{Lov78} proved the long-standing conjecture that $\chi(K(n, k)) = n - 2k + 2$ for all $k \ge 1$ and $n \ge 2k$. On the other hand, the exact value of $\chi(J(n, k))$ remains an open problem, but the clique number $\omega(J(n, k)) = \max\{k + 1, n - k + 1\}$ provides an elementary lower bound on $\chi(J(n, k))$. In fact, by writing out the well-established minimum and maximum eigenvalues of $A(J(n, k))$ \cite{Vor20}, it is easy to see that $\omega(J(n, k))$ is always equal to the Hoffman bound for $J(n, k)$. Therefore, our methods of computing the Hoffman bounds for $H(n, q)$, $J(n, k)$, and $K(n, k)$ provide no improvement over existing results, but we continue to include the $p = 1$ case---as well as the $p = \mathrm{diam}(G)$ case, where $\chi(G^p) = \chi(K_{\lvert V(G) \rvert}) = \lvert V(G) \rvert$---in all our results for the sake of generality.

In contrast, our results on graph powers of the form $H(n, q)^p$, $J(n, k)^p$, and $K(n, k)^p$ remain novel. The only existing chromatic lower bounds for powers of Hamming, Johnson, or Kneser graphs---as opposed to just the graphs themselves---regard hypercube powers specifically. The work of \cite{FFR17} is the most relevant here---in addition to lower bounds on another invariant known as the $b$-chromatic number, they determined the exact value of the clique number of hypercube powers:
\[\omega(Q_n^p) = \begin{cases}
    \sum_{i = 0}^{\frac{p}{2}} \binom{n}{i} \quad & \text{if } p \text{ is even} \\[6pt]
    \sum_{i = 0}^{\frac{p - 1}{2}} \binom{n - 1}{i} \quad & \text{if } p \text{ is odd.}
\end{cases}\]
Since $\omega$ bounds $\chi$ from below, this extends to the bound
\[\chi(Q_n^p) \ge \begin{cases}
    \sum_{i = 0}^{\frac{p}{2}} \binom{n}{i} \quad & \text{if } p \text{ is even} \\[6pt]
    \sum_{i = 0}^{\frac{p - 1}{2}} \binom{n - 1}{i} \quad & \text{if } p \text{ is odd,}
\end{cases}\]
which numerical results indicate is typically slightly tighter than our bound from Theorem \ref{thm:3.3} for hypercube powers. Using the well-known recurrence relation $\binom{n}{i} = \frac{n - i + 1}{i}\binom{n}{i - 1}$, it is straightforward to compute each $\binom{n}{i}$ in $O(1)$ from the value of $\binom{n}{i - 1}$, so computing this bound requires $O(p)$ time total---also slightly more efficient than our $O(np)$ approach. However, the present work still represents the first nontrivial, computationally feasible chromatic lower bounds for general Hamming graph powers and for any Johnson or Kneser graph powers.

These lower bounds have several concrete applications in coding theory, error correction, and distributed computing. For instance, $H(n, q)^p$ is the conflict graph for models whose processes are parametrized by $n$ parameters each taking one of $q$ values and in which conflicts occur between configurations differing in at most $p$ parameters. Since the chromatic number of a conflict graph equals the minimum number of resources needed for conflict-free allocation \cite{BNBH+98}, our lower bounds on $\chi(H(n, q)^p)$ can help establish resource requirements for distributed systems with this conflict structure.

Connections between colourings of Johnson graphs and coding theory are well known---for instance, $\chi(J(n, k))$ is precisely the minimum number of partitions needed to divide all binary words of length $n$ and weight $k$ into constant-weight codes with minimum Hamming distance $4$ \cite{EB96}. Finding tighter bounds on $\chi(J(n, k))$---whose exact value is not known---directly enables the improved construction of such codes, as demonstrated by \cite{BE11}. It is straightforward to generalize this to Johnson graph powers---$\chi(J(n, k)^p)$ is equal to the minimum number of partitions into constant-weight codes with minimum Hamming distance $2p + 2$---thus establishing the relevance of chromatic lower bounds for $J(n, k)^p$ as well.

Applications of Kneser graphs and their powers are less commonly studied than those of Hamming and Johnson graphs, but they remain relevant objects of study in several applied fields like coding theory (see, for instance, \cite{CHMv24}). While the exact value of $\chi(K(n, k))$ is well established, no computationally feasible bounds yet exist for $\chi(K(n, k)^p)$ despite the chromatic number being a fundamental graph invariant with diverse applications. The present work attempts to address this gap in the literature.

\section{Hoffman bound for Hamming graph powers}\label{section:3}

Our efficient approach to computing the Hoffman bound for $H(n, q)^p$ relies first and foremost on having a closed form for the eigenvalues of $A(H(n, q)^p)$. To this end, we first introduce the Hamming association scheme, which we shall soon see is closely linked to Hamming graphs and their spectra.

\begin{definition}[Hamming association scheme]\label{defn:8}
    The $q$-ary $n$-class \emph{Hamming association scheme} is the ordered pair $\Bigl(\Z_q^n, \bigl\{H_i^{(q)}\bigr\}_{i = 0}^n\Bigr)$, where
    \[H_i^{(q)} \coloneqq \bigl\{(\u, \vv) \in \Z_q^n \times \Z_q^n \bigm| d_H(\u, \vv) = i\bigr\}.\]
\end{definition}

Observe that we can write the Bose--Mesner algebra $\mathcal{A} \subseteq \R^{q^n \times q^n}$ of this scheme as $\mathcal{A} = \mathrm{span}_\R\bigl\{A_i^{(q)}\bigr\}_{i = 0}^n$, where the basis matrices $A_i^{(q)} \in \R^{q^n \times q^n}$ are defined by
\[\bigl[A_i^{(q)}\bigr]_{\u,\vv} \coloneqq \begin{cases}
    1 & \text{if } (\u, \vv) \in H_i^{(q)} \\
    0 & \text{if } (\u, \vv) \notin H_i^{(q)}.
\end{cases}\]
A final mathematical construction we require is the concept of Kravchuk polynomials---a family of HOPs closely related to the Hamming scheme.

\begin{definition}[Kravchuk polynomials \cite{Kra29}]\label{defn:9}
    Let $n, q \ge 1$. For all $0 \le i \le n$, we define the $i$\textsuperscript{th} \emph{Kravchuk polynomial} of $x$ with respect to $n$ and $q$ as
    \[\K_i(x; n, q) \coloneqq \sum_{j = 0}^i (-1)^j (q - 1)^{i - j} \binom{x}{j} \binom{n - x}{i - j}.\]
\end{definition}

We are now in possession of all the tools necessary to prove the following result about the spectrum of $A(H(n, q)^p)$.

\begin{lemma}\label{lemma:3.1}
    For all $n \ge 1$, $q \ge 2$, and $1 \le p \le \mathrm{diam}(H(n, q)) = n$, the maximum eigenvalue of the adjacency matrix $A(H(n, q)^p)$ is given by
    \[\sum_{i = 1}^p (q - 1)^i \binom{n}{i}\]
    and the other $n$ eigenvalues by
    \[\left\{\sum_{i = 1}^p \K_i(k; n, q)\, \middle|\, 1 \le k \le n\right\}.\]
    Moreover, the minimum eigenvalue of $A(H(n, q)^p)$ is negative.
\end{lemma}

\begin{proof}
    Suppose that $n \ge 1$, $q \ge 2$, and $1 \le p \le n$, and consider the basis matrices $\bigl\{A_i^{(q)}\bigr\}_{i = 0}^n$ of the Hamming scheme $\Bigl(\Z_q^n, \bigl\{H_i^{(q)}\bigr\}_{i = 0}^n\Bigr)$. By definition, $A_i^{(q)}$ is the distance-$i$ matrix of the Hamming graph $H(n, q)$, so the adjacency matrix of $H(n, q)^p$ can be expressed as a sum of these basis matrices:
    \[A(H(n, q)^p) = \sum_{i = 1}^p A_i^{(q)}.\]
    Now let $V = V_0 \oplus V_1 \oplus \cdots \oplus V_n$ be a common eigenspace decomposition of the $A_i^{(q)}$'s, where $V_0$ is spanned by the all-ones vector. From \cite[pp. 209--210]{BI84}, we also know that the eigenvalue $\lambda_k$ of $A_i^{(q)}$ on the eigenspace $V_k$ is given by evaluating the Kravchuk polynomial $\K_i(k; n, q)$; it then follows immediately that the $n + 1$ distinct eigenvalues of $A(H(n, q)^p)$ are
    \[\mathrm{Spec}(A(H(n, q)^p)) = \left\{\sum_{i = 1}^p \K_i(k; n, q)\, \middle|\, 0 \le k \le n\right\}.\]
    Since $V_0 = \langle \1 \rangle$ as noted above, applying the Perron--Frobenius theorem to $A(H(n, q)^p)$ tells us that the associated eigenvalue
    \[\lambda_0 = \sum_{i = 1}^p \K_i(0; n, q) = \sum_{i = 1}^p (q - 1)^i \binom{n}{i}\]
    must be the largest.

    Now observe that since the $A_i^{(q)}$'s are real symmetric matrices with all-zero diagonals, the following constraint must hold:
    \[\tr\bigl(A(H(n, q)^p)\bigr) = \sum_{i = 1}^p \tr\bigl(A_i^{(q)}\bigr) = 0.\]
    Since our expression for $\lambda_0$ necessarily evaluates to a positive value and the trace equals the sum of all eigenvalues (with multiplicities), the minimum eigenvalue of $A(H(n, q)^p)$ must be negative to satisfy the trace requirement, and we are done.
\end{proof}

\begin{remark}\label{remark:3.2}
Alternatively, one can prove Lemma \ref{lemma:3.1} via Fourier analysis on the group $\Z_q^n$. The basis matrices $A_i^{(q)}$ of the Hamming scheme act by convolution on the set of all $\Z_q^n \to \C$ functions, so they are simultaneously diagonalized by the character basis of $\Z_q^n$. Although we omit the details here for the sake of brevity, it is then straightforward to see that the corresponding eigenvalues are given by evaluating Kravchuk polynomials, precisely as in the original proof of Lemma \ref{lemma:3.1}. While this approach certainly offers additional intuition, we opt instead to work in terms of association schemes throughout the paper, which extends more naturally to our exploration of Johnson and Kneser graphs in Section \ref{section:4}. (We briefly discuss further avenues of representation-theoretic and harmonic-analytical investigation later in Section \ref{section:5}.)
\end{remark}

From here, we quickly obtain a closed form for the Hoffman bound for $H(n, q)^p$, presented below.

\begin{theorem}\label{thm:3.3}
    For all $n \ge 1$, $q \ge 2$, and $1 \le p \le \mathrm{diam}(H(n, q)) = n$,
    \[\chi(H(n, q)^p) \geq 1 + \dfrac{\sum_{i = 1}^p (q - 1)^i \binom{n}{i}}{-\min\limits_{1 \le k \le n}\left\{\sum_{i = 1}^p \K_i(k; n, q)\right\}}.\]
\end{theorem}

\begin{proof}
    Follows directly from Lemma \ref{lemma:3.1} and the Hoffman bound
    \[\chi(G) \ge 1 + \frac{\max\mathrm{Spec}(A(G))}{\lvert \min\mathrm{Spec}(A(G)) \rvert}\]
    from \cite{Hof03}.
\end{proof}

Even without further optimization, this closed form already enables significant improvements over performing SVD or QRD to find the eigenvalues of $A(H(n, q)^p) \in \R^{q^n \times q^n}$---both methods are $O(q^{3n})$. Specifically, na\"ive computation of the Hoffman bound in this manner requires $O(np^2)$ time: each eigenvalue is expressed as a double sum with $p$ outer iterations and up to $p + 1$ inner iterations. Although the inner term sums are not $O(1)$ to compute from scratch, we can invoke the same recurrence relation $\binom{n}{i} = \frac{n - i + 1}{i}\binom{n}{i - 1}$ from Subsection \ref{subsection:2.4} to derive successive binomial coefficients in constant time. Similarly, each power $(q - 1)^{i - j}$ can be computed from $(q - 1)^{i - j + 1}$ by a single division. This lets us compute each inner term in $O(1)$, so each eigenvalue requires $O(p \cdot (p + 1) \cdot 1) = O(p^2)$ time total. To find the minimum eigenvalue in the denominator, we compute $n$ eigenvalues total, resulting in $O(np^2)$ complexity overall. This approach also allocates no extra memory, requiring only $O(1)$ space.

Still, this time complexity of $O(np^2)$ becomes less feasible as we take higher powers of the Hamming graph. To obtain an alternative algorithm that requires only $O(np)$ time---the central result of this section---we thus identify and prove the following recurrence relations between Kravchuk polynomials.

\begin{lemma}\label{lemma:3.4}
    For all $1 \le i \le p$, the following recurrence relations between Kravchuk polynomials hold:
    \begin{enumerate}[label=(\roman*)]
        \item $\displaystyle \K_i(0; n, q) = \frac{(n - i + 1)(q - 1)}{i} \cdot \K_{i - 1}(0; n, q)$
        \item $\displaystyle \K_i(k; n, q) = \frac{(n - k - i + 1)(q - 1)}{i} \cdot \K_{i - 1}(k; n, q) - \frac{k}{i} \cdot \K_{i - 1}(k - 1; n, q)$ for all $k \ge 1$
    \end{enumerate}
\end{lemma}

\begin{proof}
    Fix $n \ge 1$ and $1 \le p \le n$, and suppose that $1 \le i \le p$ and $0 \le k \le n$. Multiplying the Kravchuk polynomial $\K_i(k; n, q)$ by $i$ gives us
    \begin{align*}
        i \cdot \K_i(k; n, q) & = ((i - j) + j) \sum_{j = 0}^i (-1)^j (q - 1)^{i - j} \binom{k}{j} \binom{n - k}{i - j} \\
        & = \sum_{j = 0}^i (-1)^j (q - 1)^{i - j} \binom{k}{j} (i - j) \binom{n - k}{i - j} + \sum_{j = 0}^i (-1)^j (q - 1)^{i - j} j \binom{k}{j} \binom{n - k}{i - j}.
    \end{align*}
    When $j = i$, the expression inside the first sum evaluates to zero, and when $j = 0$, the expression inside the second sum evaluates to zero. Therefore, we can simplify this further to
    \begin{align*}
        i \cdot \K_i(k; n, q) = \sum_{j = 0}^{i - 1} (-1)^j (q - 1)^{i - j} \binom{k}{j} (i - j) \binom{n - k}{i - j} + \sum_{j = 1}^i (-1)^j (q - 1)^{i - j} j \binom{k}{j} \binom{n - k}{i - j}.
    \end{align*}
    We can rearrange the well-known recurrence relation $\binom{a}{b} = \frac{(a - b + 1)}{b}\binom{a}{b - 1}$ to obtain the useful identity $b\binom{a}{b} = (a - b + 1)\binom{a}{b - 1}$. Setting $a = n - k$ and $b = i - j$, we notice that $i \cdot \K_i(k; n, q)$ then becomes
    \begin{align*}
        i \cdot \K_i(k; n, q) ={} & \sum_{j = 0}^{i - 1} (-1)^j (q - 1)^{i - j} \binom{k}{j} (n - k - i + j + 1) \binom{n - k}{i - j - 1} \\
        & + \sum_{j = 1}^i (-1)^j (q - 1)^{i - j} j \binom{k}{j} \binom{n - k}{i - j} \\
        ={} & (n - k - i + 1)(q - 1) \cdot \K_{i - 1}(k; n, q) \\
        & + \sum_{j = 0}^{i - 1} (-1)^j (q - 1)^{i - j} j \binom{k}{j} \binom{n - k}{i - j - 1} + \sum_{j = 1}^i (-1)^j (q - 1)^{i - j} j \binom{k}{j} \binom{n - k}{i - j}.
    \end{align*}
    If $k = 0$, then $\binom{k}{j} = 0$ and the two summations both evaluate to $0$, so dividing both sides by $i$ summarily results in the statement of Lemma \ref{lemma:3.4}(i). With this out of the way, we now consider the case that $k \ge 1$ toward proving Lemma \ref{lemma:3.4}(ii).

    Certainly, when $j = 0$, we have $(-1)^j (q - 1)^{i - j} j \binom{k}{j} \binom{n - k}{i - j - 1} = 0$, so we can simplify further to
    \begin{align*}
        i \cdot \K_i(k; n, q) ={} & (n - k - i + 1)(q - 1) \cdot \K_{i - 1}(k; n, q) \\
        & + \sum_{j = 1}^{i - 1} (-1)^j (q - 1)^{i - j} j \binom{k}{j} \binom{n - k}{i - j - 1} + \sum_{j = 1}^i (-1)^j (q - 1)^{i - j} j \binom{k}{j} \binom{n - k}{i - j}.
    \end{align*}
    By the absorption identity, $j\binom{k}{j} = k\binom{k - 1}{j - 1}$; since we have required $k \ge 1$, this quantity is well-defined. Plugging this in and reindexing with $l = j - 1$ then gives us
    \begin{align*}
        i \cdot \K_i(k; n, q) ={} & (n - k - i + 1)(q - 1) \cdot \K_{i - 1}(k; n, q) \\
        & - k\Biggl(\sum_{l = 0}^{i - 2} (-1)^l (q - 1)^{(i - 1) - l} \binom{k - 1}{l} \binom{n - k}{(i - 1) - l - 1} \\
        & \phantom{\; - k\Biggl(} + \sum_{l = 0}^{i - 1} (-1)^l (q - 1)^{(i - 1) - l} \binom{k - 1}{l} \binom{n - k}{(i - 1) - l}\Biggr).
    \end{align*}
    Isolating the $l = i - 1$ case in the second summation and combining the two resulting summations, we have
    \begin{align*}
        i \cdot \K_i(k; n, q) ={} & (n - k - i + 1)(q - 1) \cdot \K_{i - 1}(k; n, q) \\
        & - k\Biggl(\sum_{l = 0}^{i - 2} (-1)^l (q - 1)^{(i - 1) - l} \binom{k - 1}{l} {\left(\binom{n - k}{(i - 1) - l - 1} + \binom{n - k}{(i - 1) - l}\right)} \\
        & \phantom{\; - k\Biggl(} + (-1)^{i - 1} \binom{k - 1}{i - 1}\Biggr).
    \end{align*}
    We can then apply Pascal's identity $\binom{a}{b} = \binom{a - 1}{b - 1} + \binom{a - 1}{b}$, setting $a = n - k$ and $b = (i - 1) - l$, to obtain
    \begin{align*}
        i \cdot \K_i(k; n, q) ={} & (n - k - i + 1)(q - 1) \cdot \K_{i - 1}(k; n, q) \\
        & - k{\left(\sum_{l = 0}^{i - 2} (-1)^l (q - 1)^{(i - 1) - l} \binom{k - 1}{l} \binom{n - k + 1}{(i - 1) - l} + (-1)^{i - 1} \binom{k - 1}{i - 1}\right)}.
    \end{align*}
    Since both $(q - 1)^{(i - 1) - l}$ and $\binom{n - k + 1}{(i - 1) - l}$ evaluate to $1$ when $l = i - 1$, we can fold the isolated $(-1)^{i - 1} \binom{k - 1}{i - 1}$ term back into the sum to get
    \begin{align*}
        i \cdot \K_i(k; n, q) ={} & (n - k - i + 1)(q - 1) \cdot \K_{i - 1}(k; n, q) \\
        & - k{\left(\sum_{l = 0}^{i - 1} (-1)^l (q - 1)^{(i - 1) - l} \binom{k - 1}{l} \binom{n - k + 1}{(i - 1) - l}\right)} \\
        ={} & (n - k - i + 1)(q - 1) \cdot \K_{i - 1}(k; n, q) - k \cdot \K_{i - 1}({k - 1}; n, q).
    \end{align*}
    Finally, dividing both sides by $i$ yields the statement of Lemma \ref{lemma:3.4}(ii), thus completing the proof.
\end{proof}

This relationship allows us to iterate over successive values of $i = 1, 2, \ldots, p$ and dynamically compute the Kravchuk polynomial $\K_i(k; n, q)$, then sum these values up to obtain the $\lambda_k$'s. In particular, we maintain two arrays of length $n + 1$, called \texttt{krav} and \texttt{lambda}. \texttt{krav} stores the current iteration's values of $\K_i(k; n, q)$ for all $k = 0, 1, \ldots, n$ simultaneously, while \texttt{lambda} accumulates the partial sums $\sum_{j = 1}^i \K_j(k; n, q)$ after each iteration to build up the eigenvalues $\lambda_k = \sum_{i = 1}^p \K_i(k; n, q)$ (as per Lemma \ref{lemma:3.1}). Note that for all $k$, the base case is
\begin{align*}
     \K_0(k; n, q) = \sum_{j = 0}^0 (-1)^j (q - 1)^{0 - j} \binom{k}{j} \binom{n - k}{0 - j} = 1,
\end{align*}
so we initialize all entries of \texttt{krav} to $1$ before the first iteration.

When updating \texttt{krav} from iteration $i - 1$ to iteration $i$, we process positions in descending order from $k = n$ to $k = 1$, since computing $\K_i(k; n, q)$ via Lemma \ref{lemma:3.4}(ii) requires knowing not only $\K_{i - 1}(k; n, q)$ but also $\K_{i - 1}(k - 1; n, q)$ from the previous iteration. This ensures that when computing the new value $\texttt{krav}[k]$, $\texttt{krav}[k - 1]$ still contains its old value from iteration $i - 1$. (Of course, computing each $\K_i(0; n, q)$ via Lemma \ref{lemma:3.4}(i) is far simpler, as it depends solely on the value in the same $\texttt{krav}[0]$ position from the previous iteration.)

We formally present this dynamic programming approach below to conclude Section \ref{section:3}, alongside a thorough analysis of the algorithm's complexity.

\begin{algo}\label{algo:3.5}
    Let $n \ge 1$, $q \ge 2$, and $1 \le p \le \mathrm{diam}(H(n, q)) = n$. We can compute the Hoffman bound for the Hamming graph power $H(n, q)^p$ as follows:
    {
        \normalfont
        \begin{enumerate}
            \item Initialize an array, \texttt{krav}, of size $n + 1$ with all entries set to $1$.
            \item Initialize an array, \texttt{lambda}, of size $n + 1$ with all entries set to $0$.
            \item For each $i = 1, 2, \ldots, p$:
                \begin{enumerate}
                    \item For each $k = n, n - 1, \ldots, 1$:\\
                        \hspace*{\algorithmicindent}Set $\texttt{krav}[k] \gets \frac{(n - k - i + 1)(q - 1)}{i} \cdot \texttt{krav}[k] - \frac{k}{i} \cdot \texttt{krav}[k - 1]$.
                    \item Set $\texttt{krav}[0] \gets \frac{(n - i + 1)(q - 1)}{i} \cdot \texttt{krav}[0]$.
                    \item For each $k = 0, 1, \ldots, n$:\\
                        \hspace*{\algorithmicindent}Set $\texttt{lambda}[k] \gets \texttt{lambda}[k] + \texttt{krav}[k]$.
                \end{enumerate}
            \item Let $\lambda_{\text{max}} \gets \texttt{lambda}[0]$.
            \item Let $\lambda_{\text{min}} \gets \min\limits_{1 \le k \le n} \texttt{lambda}[k]$.
            \item Return $1 + \frac{\lambda_{\text{max}}}{-\lambda_{\text{min}}}$ as a lower bound on $\chi(H(n, q)^p)$.
        \end{enumerate}
    }
    This algorithm requires $O(np)$ time and $O(n)$ space.
\end{algo}

\begin{proof}
    Correctness follows immediately from Lemmas \ref{lemma:3.1} and \ref{lemma:3.4}, as explained in the paragraphs immediately preceding the algorithm. We now justify the claimed time and space complexities. (Note that we assume constant-time arithmetic and array access; if exact arithmetic with growing word-size is intended, complexity should be measured in bit operations instead.)

    It is straightforward to see that the space complexity of this algorithm is $O(n)$, as we allocate two arrays of size $n + 1$ each in Steps 1 and 2. With regard to time, the dominant cost lies in Step $3$, which contains an outer loop that executes $p$ times. In each iteration of this loop, Step 3(a) performs $n$ constant-time updates to the \texttt{krav} array, Step 3(b) performs one more update to \texttt{krav}, and Step 3(c) performs $n + 1$ updates to \texttt{lambda}. Across all $p$ iterations of Step 3, this amounts to $p(2n + 2) = O(np)$ constant-time operations, as noted in the concrete example provided toward the end of Section \ref{section:1}. Certainly, this dominates the $O(n)$ costs of initializing the \texttt{krav} and \texttt{lambda} arrays (Steps 1 and 2) and of computing $\lambda_{\text{min}}$ (Step 5), so we see that the overall time complexity is $O(np)$, as desired.
\end{proof}

Notably, both the time and space complexities of Algorithm \ref{algo:3.5} are independent of $q$, so it is no more expensive to compute the Hoffman bound for arbitrary $H(n, q)^p$ than for $Q_n^p$, regardless of the value of $q$. Although our $O(np)$ time approach is still slightly less efficient than \cite{FFR17}'s lower bound on $\chi(Q_n^p)$---which, as mentioned in Subsection \ref{subsection:2.4}, is computable in $O(p)$ time instead---it generalizes from hypercube powers to Hamming graph powers, representing the first practical lower bound on $\chi(H(n, q)^p)$ in general.

\section{Hoffman bounds for Johnson and Kneser graph powers}\label{section:4}

Similarly to what we did with Hamming graph powers in Section \ref{section:3}, we now work toward efficient computation of the Hoffman bound for Johnson and Kneser graph powers. As noted in Subsection \ref{subsection:2.2}, the Kneser graph $K(n, k)$ is precisely the distance-$k$ graph of the Johnson graph $J(n, k)$, so these two graphs share fundamental structural similarities despite their different adjacency conditions. Both can be viewed through the lens of the Johnson scheme and its associated family of HOPs just as we used the Hamming scheme to analyze Hamming graph powers; we introduce this association scheme now.

\begin{definition}[Johnson association scheme]\label{defn:10}
    The $n$-element $k$-class \emph{Johnson association scheme} is the ordered pair $\Bigl(\binom{[n]}{k}, \bigl\{J_i^{(n)}\bigr\}_{i = 0}^k\Bigr)$, where
    \[J_i^{(n)} \coloneqq \left\{(A, B) \in \binom{[n]}{k} \times \binom{[n]}{k}\, \middle|\, \lvert A \cap B \rvert = k - i\right\}.\]
\end{definition}

As we did for the Hamming scheme in Section \ref{section:3}, we express the Bose--Mesner algebra $\mathcal{B} \subseteq \R^{\binom{n}{k} \times \binom{n}{k}}$ of the Johnson scheme as the span $\mathcal{B} = \mathrm{span}_\R\bigl\{B_i^{(n)}\bigr\}_{i = 0}^k$, where the basis matrices $B_i^{(n)} \in \R^{\binom{n}{k} \times \binom{n}{k}}$ are defined by
\[\bigl[B_i^{(n)}\bigr]_{A,B} \coloneqq \begin{cases}
    1 & \text{if } (A, B) \in J_i^{(n)} \\
    0 & \text{if } (A, B) \notin J_i^{(n)}.
\end{cases}\]
Observe that $B_i^{(n)}$ is precisely the distance-$i$ matrix of the Johnson graph $J(n, k)$---as such, $A(J(n, k)^p)$ is simply the sum of the first $p$ $B_i^{(n)}$'s. On the other hand, since $K(n, k)$ is the distance-$k$ graph of $J(n, k)$, $A(K(n, k)^p)$ is precisely the sum of all $B_i^{(n)}$'s such that distance $i$ in $J(n, k)$ maps to a distance between $1$ and $p$ in $K(n, k)$ (we shall formalize this concept later in this section).

To use these basis matrices to compute the eigenvalues of $A(K(n, k)^p)$ and $A(J(n, k)^p)$ efficiently, we leverage the fact that the $B_i^{(n)}$'s share a common eigenspace decomposition. Toward this, we first define the Eberlein polynomials.

\begin{definition}[Eberlein polynomials \cite{Del73}]\label{defn:11}
    Let $k \ge 1$ and $n \ge k$. For all $0 \le i \le k$, we define the $i$\textsuperscript{th} \emph{Eberlein polynomial} of $x$ with respect to $n$ and $k$ as
    \[\Ein_i(x; n, k) \coloneqq \sum_{t = 0}^i (-1)^t \binom{x}{t} \binom{k - x}{i - t} \binom{n - k - x}{i - t}.\]
\end{definition}

These polynomials play the same role for the Johnson association scheme as the Kravchuk polynomials do for the Hamming scheme. Specifically, the eigenvalues of each basis matrix $B_i^{(n)}$ of the Johnson scheme $\Bigl(\binom{[n]}{k}, \bigl\{J_i^{(n)}\bigr\}_{i = 0}^k\Bigr)$ are given by evaluating the Eberlein polynomials $\Ein_i(x; n, k)$ for $0 \le i \le k$ \cite[p. 220]{BI84}.

The path to establishing the adjacency spectrum of $J(n, k)^p$ is already clear---we shall simply follow a similar strategy as we did in Lemma \ref{lemma:3.1}. However, establishing the adjacency spectrum of $K(n, k)^p$ requires us to first characterize which Johnson distances up to $k$ correspond to which Kneser distances, so we can express $A(K(n, k)^p)$ as a sum of the $B_i^{(n)}$'s. (We only consider Johnson distances up to $k$ because Kneser graphs $K(n, k)$ require $n > 2k$, in which case we know from Subsection \ref{subsection:2.2} that $\mathrm{diam}(J(n, k)) = \min\{k, n - k\} = k$.) We thus formalize this relationship using the Kneser graph distance formula from Proposition \ref{prop:2.4}.

\begin{lemma}\label{lemma:4.1}
    Let $k \ge 1$ and $n > 2k$, and define the function $\delta_{n,k} : \{0, 1, \ldots, k\} \to \Z_{\ge 0}$ by
    \[\delta_{n,k}(i) \coloneqq \min\left\{2{\left\lceil\frac{i}{n - 2k}\right\rceil},\, 2{\left\lceil\frac{k - i}{n - 2k}\right\rceil} + 1\right\}.\]
    Then for all Johnson/Kneser graph vertices $A, B \in \binom{[n]}{k}$, the relationship between distances in $J(n, k)$ and $K(n, k)$ is given by $d_{K(n, k)}(A, B) = \delta_{n,k}(d_{J(n, k)}(A, B))$.
\end{lemma}

\begin{proof}
    Fix $k \ge 1$ and $n > 2k$, and suppose that $A, B \in \binom{[n]}{k}$ are Johnson/Kneser graph vertices with $\lvert A \cap B \rvert = x$. By Proposition \ref{prop:2.3}, we know that $d_{J(n, k)}(A, B) = k - x$; substituting $i = k - x$ into the formula for $d_{K(n, k)}$ from Proposition \ref{prop:2.4} then gives us
    \[d_{K(n, k)}(A, B) = \min\left\{2{\left\lceil\frac{i}{n - 2k}\right\rceil},\, 2{\left\lceil\frac{k - i}{n - 2k}\right\rceil} + 1\right\} = \delta_{n,k}(i) = \delta_{n,k}(d_{J(n, k)}(A, B)),\]
    as desired.
\end{proof}

Having established the distance mapping $\delta_{n,k}$, we are now ready to characterize the adjacency spectra of both $J(n, k)^p$ and $K(n, k)^p$ in these terms.

\begin{lemma}\label{lemma:4.2}
    For all $k \ge 1$, the following results on the adjacency spectra of Johnson and Kneser graph powers hold:
    \begin{enumerate}[label=(\roman*)]
        \item For all $n \ge k$ and $1 \le p \le \mathrm{diam}(J(n, k)) = \min\{k, n - k\}$, the maximum eigenvalue of the adjacency matrix $A(J(n, k)^p)$ is given by
        \[\sum_{i = 1}^p \binom{k}{i} \binom{n - k}{i}\]
        and the other $k$ eigenvalues by
        \[\left\{\sum_{i = 1}^p \Ein_i(j; n, k)\, \middle|\, 1 \le j \le k\right\}.\]
        \item For all $n > 2k$ and $1 \le p \le \mathrm{diam}(K(n, k)) = \left\lceil\frac{k - 1}{n - 2k}\right\rceil + 1$, the maximum eigenvalue of the adjacency matrix $A(K(n, k)^p)$ is given by
        \[\sum_{i \in \delta_{n,k}^{-1}([p])} \binom{k}{i} \binom{n - k}{i}\]
        and the other $k$ eigenvalues by
        \[\left\{\sum_{i \in \delta_{n,k}^{-1}([p])} \Ein_i(j; n, k)\, \middle|\, 1 \le j \le k\right\},\]
        where $\delta_{n,k}$ is as defined in Lemma \ref{lemma:4.1} (and thus the preimage $\delta_{n,k}^{-1}([p])$ is the set of all distances in $J(n, k)$ that map to a distance in $K(n, k)$ between $1$ and $p$).
    \end{enumerate}
    Moreover, the minimum eigenvalues of both classes of graph powers are always negative.
\end{lemma}

\begin{proof}
    We first prove part (i). To this end, suppose that $k \ge 1$, $n \ge k$, and $1 \le p \le \min\{k, n - k\}$. Since $B_i^{(n)}$ is the distance-$i$ matrix of $J(n, k)$, we know that
    \[A(J(n, k)^p) = \sum_{i = 1}^p B_i^{(n)}.\]
    Next, consider any common eigenspace decomposition $V = V_0 \oplus V_1 \oplus \cdots \oplus V_k$ of the $B_i^{(n)}$'s, where $V_0$ is spanned by the all-ones vector. The eigenvalue $\lambda_j$ of $B_i^{(n)}$ on the eigenspace $V_j$ is given by the Eberlein polynomial $\Ein_i(j; n, k)$ \cite[p. 220]{BI84}; since $A(J(n, k)^p)$ is a sum of these matrices, its eigenvalues are thus sums of the corresponding Eberlein polynomials. We therefore see that the $k + 1$ distinct eigenvalues of $A(J(n, k)^p)$ are
    \[\mathrm{Spec}(A(J(n, k)^p)) = \left\{\sum_{i = 1}^p \Ein_i(j; n, k)\, \middle|\, 0 \le j \le k\right\}.\]
     By an argument nearly identical to the one presented in the proof of Lemma \ref{lemma:3.1}, it is straightforward to see that the eigenvalue
    \[\lambda_0 = \sum_{i = 1}^p \Ein_i(0; n, k) = \sum_{i = 1}^p \binom{k}{i} \binom{n - k}{i}\]
    is the largest and that the minimum eigenvalue must be negative to satisfy trace constraints on $A(J(n, k)^p)$, as desired.

    We now see that the proof of part (ii) follows quickly. Choose some new $n > 2k$ and $1 \le p \le \left\lceil\frac{k - 1}{n - 2k}\right\rceil + 1$, and note that since $K(n, k)$ is the distance-$k$ graph of $J(n, k)$, Lemma \ref{lemma:4.1} tells us immediately that
    \[A(K(n, k)^p) = \sum_{i \in \delta_{n,k}^{-1}([p])} B_i^{(n)}.\]
    The remainder of the proof proceeds identically to part (i)---except summing over $i \in \delta_{n,k}^{-1}([p])$ instead of $i \in [p]$---and we are done.
\end{proof}

\begin{remark}\label{remark:4.3}
    Since $J(n, k)$ is distance-transitive, it is actually fairly straightforward to derive its adjacency spectrum simply by looking at its automorphism group. In particular, \cite[pp. 255--256]{BCN89} provide a concise proof that the $k + 1$ distinct eigenvalues of $A(J(n, k))$ are given by $\lambda_j = (k - j)(n - k - j) - j$ with multiplicities $\mu_j = \binom{n}{j} - \binom{n}{j - 1}$, taken over $j = 0, 1, \ldots, k$. (Indeed, each $\Ein_1(j; n, k)$ evaluates precisely to $\lambda_j$ as defined above, matching our findings from Lemma \ref{lemma:4.2}(i) in the $p = 1$ case.) However, it is not easy to generalize this construction to powers of Johnson graphs (which are not, in general, distance-transitive), hence our need to invoke properties of the Johnson association scheme instead.
\end{remark}

\begin{remark}\label{remark:4.4}
    We note further that the preimage $\delta_{n,k}^{-1}([p])$ need not necessarily form a contiguous interval; for instance, direct computation shows that $\delta_{n,k}^{-1}([p]) = \{1, 4\}$ for $(n, k, p) = (9, 4, 2)$. Nevertheless, this presents no algorithmic difficulty---we can compute $\delta_{n,k}^{-1}([p])$ in $O(k)$ time simply by checking whether $1 \le \delta_{n,k}(i) \le p$ for each relevant Johnson distance $i \in \{0, 1, \ldots, k\}$. (In fact, it is easy to see that we always have $\delta_{n,k}(0) = 0$ and $\delta_{n,k}(k) = 1$, so we can automatically exclude $0$ and include $k$ as a micro-optimization.)
\end{remark}

Notwithstanding the slightly lengthier exposition required compared to Section \ref{section:3}, we have firmly secured all the prerequisites needed to obtain closed forms for the Hoffman bounds on $J(n, k)^p$ and $K(n, k)^p$, which we now present.

\begin{theorem}\label{thm:4.5}
    For all $k \ge 1$, the following chromatic lower bounds for Johnson and Kneser graph powers hold:
    \begin{enumerate}[label=(\roman*)]
        \item For all $n \ge k$ and $1 \le p \le \mathrm{diam}(J(n, k)) = \min\{k, n - k\}$,
        \[\chi(J(n, k)^p) \geq 1 + \dfrac{\sum_{i = 1}^p \binom{k}{i} \binom{n - k}{i}}{-\min\limits_{1 \le j \le k}\left\{\sum_{i = 1}^p \Ein_i(j; n, k)\right\}}.\]
        \item For all $n > 2k$ and $1 \le p \le \mathrm{diam}(K(n, k)) = \left\lceil\frac{k - 1}{n - 2k}\right\rceil + 1$,
        \[\chi(K(n, k)^p) \geq 1 + \dfrac{\sum_{i \in \delta_{n,k}^{-1}([p])} \binom{k}{i} \binom{n - k}{i}}{-\min\limits_{1 \le j \le k}\left\{\sum_{i \in \delta_{n,k}^{-1}([p])} \Ein_i(j; n, k)\right\}}.\]
    \end{enumerate}
\end{theorem}

\begin{proof}
    Analogously to Theorem \ref{thm:3.3}, the statement follows directly from Lemma \ref{lemma:4.2} and \cite{Hof03}'s classical chromatic lower bound.
\end{proof}

As with Theorem \ref{thm:3.3} and Hamming graph powers, these closed forms represent significant improvements over na\"ive eigenvalue computation with SVD or QRD. However, na\"ively computing these bounds still incurs substantial costs. For $J(n, k)^p$, we must sum $p$ Eberlein polynomials to compute each of the $k + 1$ eigenvalues, with each $\Ein_i(j; n, k)$ itself taking $O(i)$ time to compute even with standard binomial recurrences, resulting in $O(kp^2)$ time total. For $K(n, k)^p$, the situation is similarly challenging: since $\bigl\lvert \delta^{-1}([p]) \bigr\rvert$ can grow as large as $O(k)$ and thus each Eberlein polynomial evaluation costs $O(i)$ for $i$ as large as $k$, computing all $k + 1$ eigenvalues requires $O(k^3)$ time overall. To obtain more efficient $O(kp)$ and $O(k^2)$ algorithms for Johnson and Kneser graph powers, respectively, we therefore cite a well-known recurrence relation between Eberlein polynomials (in particular, one slightly more complex than the relations between Kravchuk polynomials in Lemma \ref{lemma:3.4}, involving three terms instead of two).

\begin{lemma}\label{lemma:4.6}
    Let $k \ge 1$ and $n \ge k$, and for all $1 \le i \le k$ and $0 \le j \le k$, define the quantities $\alpha_i \coloneqq (n - k - i + 1)(k - i + 1)$ and $\beta_j \coloneqq j(n - j + 1)$. Then the following recurrence relations between Eberlein polynomials hold:
    \begin{enumerate}[label=(\roman*)]
        \item $\displaystyle \Ein_i(0; n, k) = \frac{\alpha_i}{i^2} \cdot \Ein_{i - 1}(0; n, k)$ for all $1 \le i \le k$
        \item $\displaystyle \Ein_i(j; n, k) = \frac{\alpha_i - \beta_j + (i - 1)^2}{i^2} \cdot \Ein_{i - 1}(j; n, k) - \frac{\alpha_{i - 1}}{i^2} \cdot \Ein_{i - 2}(j; n, k)$ for all $0 \le j \le k$ and $2 \le i \le k$
    \end{enumerate}
\end{lemma}

\begin{proof}
    An equivalent formulation of these relations is given in \cite[p. 35]{KS98} (where the Eberlein polynomials are referred to instead as ``dual Hahn'' polynomials).
\end{proof}

Whereas the recurrence relations between Kravchuk polynomials given in Lemma \ref{lemma:3.4} expressed $\K_i$ solely in terms of $\K_{i - 1}$, Lemma \ref{lemma:4.6} expresses $\Ein_i$ in terms of both $\Ein_{i - 1}$ and $\Ein_{i - 2}$. Therefore, exploiting these recurrence relations to build a dynamic programming algorithm requires us to maintain values from two previous iterations simultaneously, rather than just one as in Algorithm \ref{algo:3.5}. To this end, we maintain three arrays of length $k + 1$: \texttt{eberPrev} stores the Eberlein polynomials from iteration $i - 2$, \texttt{eberCurr} stores those from iteration $i - 1$, and \texttt{lambda} accumulates the partial sums to build up the eigenvalues. After each iteration, the old \texttt{eberCurr} becomes the new \texttt{eberPrev}, and the newly computed values become the new \texttt{eberCurr}. Of course, the base cases once again require special handling: for all $j$, we have
\[\mathcal{E}_0(j; n, k) = \sum_{t = 0}^0 (-1)^t \binom{j}{t} \binom{k - j}{0 - t} \binom{n - k - j}{0 - t} = 1,\]
and for $i = 1$, we can compute $\mathcal{E}_1(j; n, k) = \alpha_1 - \beta_j$ directly via Lemma \ref{lemma:4.6}(i).

Now in possession of all the necessary prerequisites, we present the algorithm for Johnson graph powers below, with the intention of adapting it afterwards for Kneser graph powers.

\begin{algo}\label{algo:4.7}
    Let $k \ge 1$, $n \ge k$, and $1 \le p \le \mathrm{diam}(J(n, k)) = \min\{k, n - k\}$. Using the same definitions of the $\alpha_i$'s and $\beta_j$'s from Lemma \ref{lemma:4.6}, we can compute the Hoffman bound for the Johnson graph power $J(n, k)^p$ as follows:
    {
        \normalfont
        \begin{enumerate}
            \item Initialize an array, \texttt{eberPrev}, of size $k + 1$ with all entries set to $1$.
            \item Initialize an array, \texttt{eberCurr}, of size $k + 1$.
            \item Initialize an array, \texttt{temp}, of size $k + 1$ (to facilitate updating \texttt{eberPrev} and \texttt{eberCurr}).
            \item Initialize an array, \texttt{lambda}, of size $k + 1$ with all entries set to $0$.
            \item For each $i = 1, 2, \ldots, p$:
                \begin{enumerate}
                    \item If $i = 1$:
                        \begin{enumerate}
                            \item For each $j = 0, 1, \ldots, k$:\\
                                \hspace*{\algorithmicindent}Set $\texttt{eberCurr}[j] \gets \alpha_1 - \beta_j$.
                        \end{enumerate}
                    \item Else (i.e., $i \ge 2$):
                        \begin{enumerate}
                            \item For each $j = 0, 1, \ldots, k$:\\
                                \hspace*{\algorithmicindent}Set $\texttt{temp}[j] \gets \frac{\alpha_i - \beta_j + (i - 1)^2}{i^2} \cdot \texttt{eberCurr}[j] - \frac{\alpha_{i - 1}}{i^2} \cdot \texttt{eberPrev}[j]$.
                            \item Copy the values in \texttt{eberCurr} into \texttt{eberPrev}.
                            \item Copy the values in \texttt{temp} into \texttt{eberCurr}.
                        \end{enumerate}
                    \item For each $j = 0, 1, \ldots, k$:\\
                        \hspace*{\algorithmicindent}Set $\texttt{lambda}[j] \gets \texttt{lambda}[j] + \texttt{eberCurr}[j]$.
                \end{enumerate}
            \item Let $\lambda_{\text{max}} \gets \texttt{lambda}[0]$.
            \item Let $\lambda_{\text{min}} \gets \min\limits_{1 \le j \le k} \texttt{lambda}[j]$.
            \item Return $1 + \frac{\lambda_{\text{max}}}{-\lambda_{\text{min}}}$ as a lower bound on $\chi(J(n, k)^p)$.
        \end{enumerate}
    }
    This algorithm requires $O(kp)$ time and $O(k)$ space.
\end{algo}

\begin{proof}
    Analogously to the verification of Algorithm \ref{algo:3.5}, correctness follows from Lemmas \ref{lemma:4.2}(i) and \ref{lemma:4.6}. We now justify our complexity claims, again assuming constant-time arithmetic and array access.
    
    The space complexity is clearly $O(k)$, as we allocate four arrays of size $k + 1$ in Steps 1 through 4. In terms of time complexity, Step 5 dominates with its outer loop executing precisely $p$ times. Within each iteration, Step 5(a) handles the $i = 1$ base case with $k + 1$ simple array updates and arithmetic operations, Steps 5(b)(i)--(iii) handle $i \ge 2$ with $k + 1$ updates plus two array copies of $k + 1$ elements each, and Step 5(c) again performs $k + 1$ constant-time operations. Across all $p$ iterations, this amounts to $O(kp)$ operations total, dominating the $O(k)$ initialization (Steps 1--4) and $O(k)$ cost of computing $\lambda_{\text{min}}$ (Step 7). Therefore, the overall time complexity is $O(kp)$, as desired.
\end{proof}

The algorithm for Kneser graph powers proceeds nearly identically to Algorithm \ref{algo:4.7}, with the key difference being that we must iterate through all Johnson distances from $1$ to $k$ (instead of from $1$ to $p$ specifically) to properly utilize Lemma \ref{lemma:4.6}(ii), regardless of which distances actually contribute to the Kneser graph power. As noted in Remark \ref{remark:4.4}, we first precompute the preimage $\delta^{-1}([p])$ in $O(k)$ time, then conditionally accumulate Eberlein polynomial values into the eigenvalue array only when the current distance $i$ belongs to said preimage.

\begin{algo}\label{algo:4.8}
    Let $k \ge 1$, $n > 2k$, and $1 \le p \le \mathrm{diam}(K(n, k)) = \left\lceil\frac{k - 1}{n - 2k}\right\rceil + 1$. Using the same definitions of the $\alpha_i$'s, the $\beta_j$'s, and $\delta_{n,k}$ from Lemmas \ref{lemma:4.6} and \ref{lemma:4.1}, we can compute the Hoffman bound for the Kneser graph power $K(n, k)^p$ as follows:
    {
        \normalfont
        \begin{enumerate}
            \item Compute $S \gets \{i \in \{1, 2, \ldots, k\} \mid 1 \le \delta_{n,k}(i) \le p\}$ (the preimage $\delta_{n,k}^{-1}([p])$).
            \item Initialize an array, \texttt{eberPrev}, of size $k + 1$ with all entries set to $1$.
            \item Initialize an array, \texttt{eberCurr}, of size $k + 1$.
            \item Initialize an array, \texttt{temp}, of size $k + 1$ (to facilitate updating \texttt{eberPrev} and \texttt{eberCurr}).
            \item Initialize an array, \texttt{lambda}, of size $k + 1$ with all entries set to $0$.
            \item For each $i = 1, 2, \ldots, k$:
                \begin{enumerate}
                    \item If $i = 1$:
                        \begin{enumerate}
                            \item For each $j = 0, 1, \ldots, k$:\\
                                \hspace*{\algorithmicindent}Set $\texttt{eberCurr}[j] \gets \alpha_1 - \beta_j$.
                        \end{enumerate}
                    \item Else (i.e., $i \ge 2$):
                        \begin{enumerate}
                            \item For each $j = 0, 1, \ldots, k$:\\
                                \hspace*{\algorithmicindent}Set $\texttt{temp}[j] \gets \frac{\alpha_i - \beta_j + (i - 1)^2}{i^2} \cdot \texttt{eberCurr}[j] - \frac{\alpha_{i - 1}}{i^2} \cdot \texttt{eberPrev}[j]$.
                            \item Copy the values in \texttt{eberCurr} into \texttt{eberPrev}.
                            \item Copy the values in \texttt{temp} into \texttt{eberCurr}.
                        \end{enumerate}
                    \item If $i \in S$:
                        \begin{enumerate}
                            \item For each $j = 0, 1, \ldots, k$:\\
                                \hspace*{\algorithmicindent}Set $\texttt{lambda}[j] \gets \texttt{lambda}[j] + \texttt{eberCurr}[j]$.
                        \end{enumerate}
                \end{enumerate}
            \item Let $\lambda_{\text{max}} \gets \texttt{lambda}[0]$.
            \item Let $\lambda_{\text{min}} \gets \min\limits_{1 \le j \le k} \texttt{lambda}[j]$.
            \item Return $1 + \frac{\lambda_{\text{max}}}{-\lambda_{\text{min}}}$ as a lower bound on $\chi(K(n, k)^p)$.
        \end{enumerate}
    }
    This algorithm requires $O(k^2)$ time and $O(k)$ space.
\end{algo}

\begin{proof}
    Again, correctness follows from Lemmas \ref{lemma:4.2}(ii) and \ref{lemma:4.6}. With regard to complexity, the algorithm is nearly identical to Algorithm \ref{algo:4.7}, with the only main differences being that Step 1 precomputes $\delta_{n,k}^{-1}([p])$ in $O(k)$ time, and that Step 6(c) conditionally accumulates into \texttt{lambda} only when $i \in \delta_{n,k}^{-1}([p])$. Since Step 6 iterates through all $i$ from $1$ to $k$, we have $k$ iterations in the outer loop, each performing $O(k)$ work and thus yielding $O(k^2)$ time complexity regardless of the value of $p$. The space complexity still remains $O(k)$ from the four arrays of size $k + 1$ allocated in Steps 1--5.
\end{proof}

As with Algorithm \ref{algo:3.5} for Hamming graph powers, Algorithms \ref{algo:4.7} and \ref{algo:4.8} both achieve time complexities that are independent of $n$, which is even more striking this time around given that both Johnson and Kneser graphs grow exponentially in vertex count as $n$ increases. Algorithm \ref{algo:4.7} achieves $O(kp)$ time complexity for Johnson graph powers, which becomes $O(k^2)$ in the worst case when $p$ approaches $\min\{k, n - k\}$. In contrast, Algorithm \ref{algo:4.8} requires $O(k^2)$ time regardless of the value of $p$, since we must compute Eberlein polynomials up to index $k$ even when $\delta_{n,k}^{-1}([p])$ contains fewer than $k$ elements. In the end, both algorithms represent dramatic computational advantages over na\"ive methods, analogous to the improvements achieved for Hamming graph powers in Section \ref{section:3}.

\section{Conclusion and future work}\label{section:5}

Motivated by applications of (powers of) Hamming, Johnson, and Kneser graphs (especially the former two) to coding theory and distributed computing, as well as the general mathematical significance of these objects, we investigated methods of efficiently computing \cite{Hof03}'s eigenvalue-based chromatic lower bound for $H(n, q)^p$, $J(n, k)^p$, and $K(n, k)^p$ using recurrence relations between associated HOPs. In particular, we used Kravchuk and Eberlein polynomials to compute the adjacency eigenvalues of $H(n, q)^p$, $J(n, k)^p$, and $K(n, k)^p$ in $O(np)$, $O(kp)$, and $O(k^2)$ time, respectively---a marked improvement over na\"ive $O(q^{3n})$ and $O\bigl(\binom{n}{k}^3\bigr)$ approaches like SVD or QRD. As such, the present work represents the first nontrivial, computationally feasible chromatic lower bounds for powers of Hamming, Johnson, and Kneser graphs.

For future work, we consider leveraging our efficient eigenvalue computations to calculate more recent spectral bounds on $\chi$, such as \cite{WE13}'s (which uses all adjacency eigenvalues, not just the minimum and maximum), and comparing the tightness of these bounds to the Hoffman bounds for Hamming, Johnson, and Kneser graph powers. There also exist spectral bounds on other graph invariants for which our efficient eigenvalue computations might prove useful. Among others, Delsarte's linear programming bound on the independence number \cite[p. 31]{Del73} and more recent bounds on the fractional chromatic number \cite{GS24} may be relevant here. It may also be worthwhile to investigate other highly symmetric, vertex-transitive graphs for which computationally practical chromatic lower bounds are not known, namely those with corresponding association schemes whose basis matrices have eigenvalues expressible in terms of HOPs. Natural candidates include Grassmann graphs, which form another prominent family of graphs with known applications to coding theory \cite{Mog22} and whose adjacency eigenvalues are given by $q$-Hahn polynomials.

Finally, we remark that our results can be understood within a broader representation-theoretic framework. Although we work exclusively with association schemes in this paper, the deeper structure underlying our approach is the presence of multiplicity-free permutation representations---for Hamming graph powers, the wreath product $S_q \wr S_n$ acting on $\Z_q^n$ (as briefly alluded to in Remark \ref{remark:3.2}), and for Johnson and Kneser graph powers, the symmetric group $S_n$ acting on $\binom{[n]}{k}$. (See \cite{CSST08} for more background on the theoretical foundations underpinning these structures.) Such representations ensure that the associated Bose--Mesner algebras are commutative, which in turn guarantees that their basis matrices are simultaneously diagonalizable (a convenient property of which we repeatedly make use). The association scheme framework we have used thus far guarantees these properties, but they may arise more generally from transitive group actions. This motivates the question of whether graphs with multiplicity-free automorphism groups, even those without distance regularity or complete association schemes, might similarly have adjacency eigenvalues conveniently expressible in terms of HOPs.

\section*{Acknowledgements}

F.A.S. thanks Rajesh Pereira for many helpful conversations regarding
algebraic combinatorics and Jack Bedard for discussions about general
graph theory. L.M.B.V. thanks Nathaniel Johnston and Laurie Ricker
for general advice, Liam Keliher for helpful conversations regarding
combinatorics, and Tanner Altenkirk for feedback on the paper's structure.

\bibliographystyle{amsalpha}
\bibliography{refs}

\end{document}